\documentclass[11pt]{article}

\usepackage{epsfig}
\usepackage{amssymb, latexsym}
\usepackage{amscd}
\usepackage[all]{xy}
\usepackage{epsf}
\usepackage[mathscr]{eucal}

\usepackage{amsmath,amsfonts,amscd,amssymb,amsthm}

\usepackage{appendix}

\long\def\comment#1\endcomment{}

\theoremstyle{plain}

\newtheorem{theorem}{{\sc Theorem}}[section]
\newtheorem{lemma}[theorem]{\sc Lemma}

\newtheorem{prop}[theorem]{\sc Proposition}
\newtheorem{coroll}[theorem]{\sc Corollary}

\theoremstyle{plain}

\newtheorem{defn}[theorem]{\sc Definition}

\theoremstyle{exercise}

\makeatletter \@addtoreset{equation}{section} \makeatother

\def\eqref#1{\thetag{\ref{#1}}}

\let\latexref=\ref
\def\ref#1{{\normalfont{\latexref{#1}}}}

\newcommand{\hdot}{{\:\raisebox{3pt}{\text{\circle*{1.5}}}}}
\def\dlim_#1{{\displaystyle\lim_{#1}}^\hdot}

\newcommand{\id}{\operatorname{\rm id}}

\newcommand{\eqto}{\mathrel{\stackrel{\sim}{\to}}}

\newcommand{\Mor}{\mathrm{Mor}}
\newcommand{\Ob}{\mathrm{Ob}}

\newcommand{\opp}{\mathrm{opp}}

\newcommand{\fin}{\mathrm{fin}}
\newcommand{\Mon}{\boldsymbol{\mathrm{Mon}}}

\newcommand{\Ho}{\mathrm{Ho}}

\newcommand{\Id}{\mathrm{Id}}

\newcommand{\mon}{\mathrm{mon}}
\newcommand{\Sets}{\mathbf{Sets}}

\newcommand{\maxx}{\mathrm{max}}

\newcommand{\sevafigc}[4]{\begin{figure}[h]\centerline{
 \epsfig{file=#1,width=#2,angle=#3}}
\bigskip\caption{#4}\end{figure}}

\title{\sc{A bialgebra axiom\\ and the Dold-Kan correspondence}}

\author{\sc{Boris Shoikhet}}
\date{}

\begin{document}\maketitle
{\footnotesize
\begin{center}{\parbox{4,5in}{{\sc Abstract.} We introduce a {\it bialgebra axiom} for a pair $(c,\ell)$ of a colax-monoidal and a lax-monoidal structures on a functor $F\colon \mathscr{M}_1\to \mathscr{M}_2$ between two (strict) symmetric monoidal categories. This axiom can be regarded as a weakening of the property of $F$ to be a strict symmetric monoidal functor. We show that this axiom transforms well when passing to the adjoint functor or to the categories of monoids. Rather unexpectedly, this axiom holds for the Alexander-Whitney colax-monoidal and the Eilenberg-MacLane lax-monoidal structures on the normalized chain complex functor in the  Dold-Kan correspondence. This fact, proven in Section 2, opens up a way for many applications, which we will consider in our sequel paper(s).}}
\end{center}
}

\bigskip
\bigskip

\section*{\sc Introduction}
\subsection{\sc }
This paper appeared in the author's attempt to understand the theory of weak monoidal Quillen pairs of Schwede and Shipley [SchS03].
In that paper, the authors have two symmetric monoidal categories $\mathscr{M}_1$ and $\mathscr{M}_2$, both with compatible closed model structures, and study when a Quillen equivalence $L\colon \mathscr{M}_1\rightleftarrows \mathscr{M}_2\colon R$ defines an equivalence on the homotopy categories of monoids $\Ho\Mon\mathscr{M}_1\eqto\Ho\Mon\mathscr{M}_2$.

The question becomes more complicated when one does not suppose that either $L$ or $R$ is a strict monoidal functor, but they have weaker (co)lax-monoidal structures. This complication is not an artificial one, as it is the case with the main illustrating examples, the Dold-Kan correspondence. The Dold-Kan correspondence is an adjoint equivalence
$$
N \colon \mathscr{M}od(\mathbb{Z})^\Delta\rightleftarrows \mathscr{C}(\mathbb{Z})^-\colon \Gamma
$$
between the categories of simplicial abelian groups and non-positively graded complexes of abelian groups.

The both categories are symmetric monoidal, but the Dold-Kan correspondence is not compatible with the symmetric monoidal structures. In fact, the normalized chain complex functor admits the Alexander-Whitney colax-monoidal structure and the Eilenberg-MacLane shuffle lax-monoidal structure.

A functor $R\colon \mathscr{M}_2\to \mathscr{M}_1$ between symmetric monoidal categories, and a lax-monoidal structure on $R$, define a functor $R^\mon\colon \Mon\mathscr{M}_2\to\Mon \mathscr{M}_1$ between their categories of monoids, in a natural way. A colax-monoidal structure on $F$ is used to construct a lax-monoidal structure on the left adjoint to $R$ (it is exists), and finally to construct a left adjoint to $R^\mon$.

The author's intention was to separate the ``linear algebra picture'', consisting of these various adjoint with (co)lax-monoidal structures on them, with the ``homotopy picture'', given by closed model structures.

In the ``linear algebra part'' of the picture the both categories $\Mon\mathscr{M}_1$ and $\Mon\mathscr{M}_2$ are symmetric monoidal again, and the functor
$R^\mon\colon \mathscr{M}_2\to \mathscr{M}_1$ is always lax-monoidal. To iterate the the linear part picture, one needs to have a colax-monoidal structure on $R^\mon$.

The following question is very natural from this point of view:

\begin{equation}
\parbox{5,5in}{
What compatibily the lax-monoidal and the colax-monoidal structures on $R$ should have, in order to induce naturally a colax-monoidal structure on $R^\mon$, which, together with its natural lax-monoidal structure, again obeys this compatibility?
}
\end{equation}

When this question is formulated, it is not hard to find the answer, expressed in the {\it bialgebra axiom} (see Section \ref{bialg}). This axiom is some compatibility on the pair $(c_F,\ell_F)$ of colax-monoidal and lax-monoidal structures between symmetric monoidal categories.

As well, it is not hard to show that this axiom has nice functorial behavior. This is done in Section 1.

Concerning the terminology, recall that in algebra {\it a (co)associative bialgebra} over a field $k$ is a $k$-vector space $A$, endowed with two structures: of a product $m\colon A\otimes A\to A$, and of a coproduct $\Delta\colon A\to A^{\otimes 2}$, subject to following 3 axioms:
\begin{itemize}
\item[(i)] the associativity of the product $m$,
\item[(ii)] the coassociativity of the coproduct $\Delta$,
\item[(iii)] the compatibility: $\Delta(a*b)=\Delta(a)*\Delta(b)$, with $a*b=m(a\otimes b)$.
\end{itemize}

When the associativity and the coassociativity are at the origin of the lax- and colax-monoidality, our bialgebra axiom is an analog of the compatibility axiom.

The compatibility axiom can be drawn as in Figure 1 below.

\sevafigc{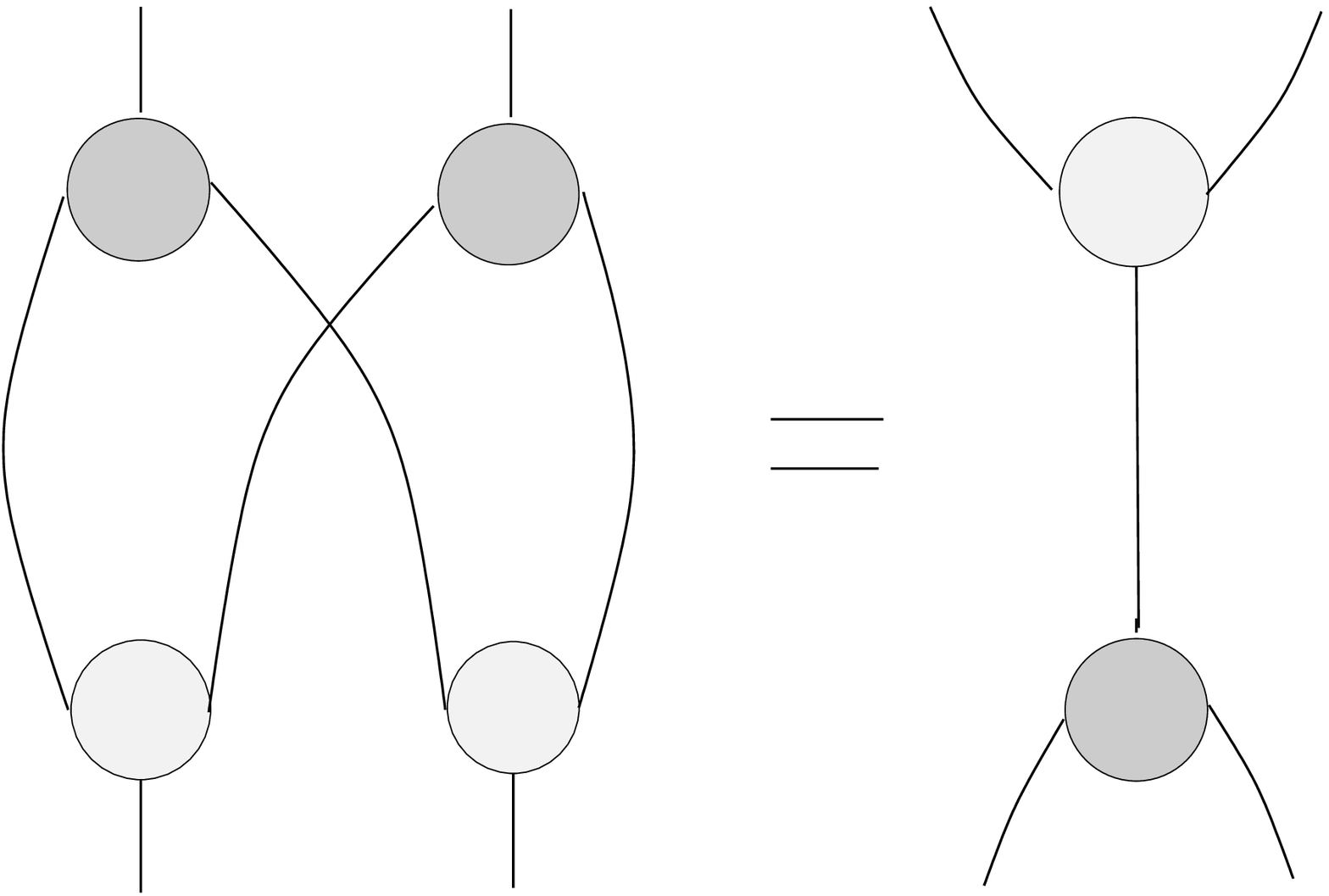}{40mm}{0}{The compatibility axiom in (co)associative bialgebra}

\subsection{\sc}
What is rather amazing, that this axiom holds for the pair of the Alexander-Whitney colax structure and the Eilenberg-MacLane shuffle lax structure on the functor $N$ of normalized chain complex, in the Dold-Kan correspondence. Even more unexpectedly, the bialgebra axiom holds as well on the chain complex level, not only mod out the degenerate simplices. Let us explain why we call this ``amazing''.

Some very simple case, a test case when the bialgebra axiom holds, is the following:

 Let $F\colon \mathscr{M}_1\to\mathscr{M}_2$ be a functor between two (strict) symmetric monoidal categories, and $c_F\colon F(X\otimes Y)\to F(X)\otimes F(Y)$, $\ell_F\colon F(X)\otimes F(Y)\to F(X\otimes Y)$ are (co)lax-monoidal structures, such that
\begin{itemize}
\item[(i)] the both $c_F$ and $\ell_F$ are symmetric,
\item[(ii)] the both compositions
$$
F(X\otimes Y)\xrightarrow{c_F}F(X)\otimes F(Y)\xrightarrow{\ell_F}F(X\otimes Y)
$$
and
$$
F(X)\otimes F(Y) \xrightarrow{\ell_F} F(X\otimes Y)\xrightarrow{c_F}F(X)\otimes F(Y)
$$
are identity maps.
\end{itemize}
Then the bialgebra axiom holds by trivial reasons, see Lemma 1.3.

Suppose now that $\mathscr{M}_1=\mathscr{M}od(\mathbb{Z})^\Delta$, $\mathscr{M}_2=\mathscr{C}(\mathbb{Z})^-$, $F=N$ is the normalized chain complex functor, $c_F$ is the Alexander-Whitney map, and $\ell_F$ is the Eilenberg-MacLane map.

Then each of the four conditions above is known to hold only up-to-homotopy, see Proposition \ref{before}. Thus, from general reasons, one can only expect that the bialgebra axiom holds only up-to-homotopy, which would make more difficult to track its functorial behavior.

Due to our Main Theorem \ref{main}, proven in Section 2.2, the bialgebra axiom in this case holds strictly.
That is, this axiom, being a finer statement, makes a reasonable compromise between the rigidity and the reality.

\subsection{\sc Simplicial sets}
For convenience of the reader, and to fix the notations, recall here the definition of a simplicial set.

The category $\Delta$ has objects $[0],[1],[2],\dots$, and a morphism $f:[m]\to[n]$ is a map of sets $f\colon \{0<1<2\dots <m\}\to \{0<1<2\dots <n\}$
such that $f(i)\le f(j)$ whenever $i\le j$.
A {\it simplicial set} is a functor $X\colon\Delta^\opp\to\Sets$. More directly, it is a collection of sets $\{X_n\}$ with operators
$$
d_i\colon X_n\to X_{n-1},\ \ i=0,\dots, n
$$
and
$$
s_i\colon X_n\to X_{n+1},\ \ i=0,\dots,n
$$
obeying the {\it simplicial identities}:

\begin{equation}\label{simplicial}
\begin{aligned}
\ d_id_j&=d_{j-1}d_i,\ &i<j,\\
s_is_j&=s_{j+1}s_i,\ &i\le j,\\
d_is_j&=s_{j-1}d_i,\ &i<j,\\
\ &=1,\  &i=j,\  i=j+1,\\
\ &=s_jd_{i-1},\ &i>j+1
\end{aligned}
\end{equation}

The operator $d_i\colon X_n\to X_{n-1}$ is corresponded to the strict embedding map $\epsilon^i\colon [n-1]\to [n]$ whose image does not contain $i$, and the operaror $s_i\colon X_n\to X_{n+1}$ is corresponded to the (non-strictly) monotonous map $\eta^i\colon [n+1]\to [n]$ such that the only point in $[n]$ with double pre-image is $i$. One has $d_i=X(\epsilon^i)$ and $s_i=X(\eta^i)$.

A simplex $\sigma$ in a simplicial set is called {\it degenerate} if it has a form $\sigma=s_i\sigma^\prime$ for some $i$.

\begin{prop}
Any morphism $\mu\colon [q]\to[p]$ in $\Delta$ can be uniquely written in form
\begin{equation}
\mu=\epsilon^{i_s}\dots \epsilon^{i_1}\eta^{j_1}\dots\eta^{j_t}
\end{equation}
with $0\le i_1<\dots <i_s\le p$, $0\le j_1<\dots<j_t<q$, and $q-t+s=p$.
\end{prop}

\subsection{\sc Organization of the paper}
Our general intention was to collect all ``big'' commutative diagrams together at the end of the paper, in Section 3 called ``Diagrams''.
We refer each time when we need some ``big'' diagram, to the corresponding part of this Section. In particular, the bialgebra axiom is defined in Section \ref{bialg}.

In Section 1 we study the functorial properties of the bialgebra axiom. Those are for adjoint functor, in Section \ref{section12}, and to the categories of monoids, in Section \ref{section13}.

In Section 2 we introduce the Dold-Kan correspondence, and prove the bialgebra axiom for the pair of the Alexander-Whitney colax structure and the Eilenberg-MacLane shuffle lax structure on the normalized chain complex functor. This is the most unexpected point of our story, from general considerations, one could only conclude that the bialgebra axiom holds up to a homotopy.

The most applications are left beyond this paper (which can be regarded as the first paper of overall project).
The theory developed here will be applied to a construction of colax-monoidal cofibrant resolutions in our next paper.
These resolutions are very useful for working with Leinster's definition [Le] of weak Segal monoids in the $k$-linear (or, more generally, non-cartesian-monoidal) setting. The latter is necessary, for instance, to extend the Kock-To\"{e}n approach [KT] to the Deligne conjecture from simplicial to $k$-linear context.

\subsection*{\sc}
\subsubsection*{\sc Acknowledgements}
I am grateful to Stefan Schwede for valuable discussions on the subject of the paper.
I am indebted to Martin Schlichenmaier for his kindness and support during my 5-year appointment at the University of Luxembourg, which made possible my further development as a mathematician.
The work was done during research stay at the Max-Planck Institut f\"{u}r Mathematik, Bonn. I am thankful to the MPIM for hospitality, financial support, and for very creative working atmosphere.

\section{\sc The bialgebra axiom and its functorial properties}
\subsection{\sc Generalities}
Let $\mathscr{C}$ and $\mathscr{D}$ be two strict monoidal categories, and
let $F\colon \mathscr{C}\to\mathscr{D}$ be a functor with two properties:
\begin{itemize}
\item[1)] $F$ is an equivalence of the underlying categories,
\item[2)] $F$ is monoidal.
\end{itemize}
Then one can choose a quasi-inverse to $F$ functor $G_1\colon \mathscr{D}\to \mathscr{C}$ such that $(F,G_1)$ is an adjoint equivalence, and we can choose another quasi-inverse to $F$ functor $G_2\colon \mathscr{D}\to\mathscr{C}$ which is monoidal (see ???).
In the same time, one may not choose a quasi-inverse $G$ enjoying the both properties simultaneously. This is how lax-monoidal and colax-monoidal functors appear. In this Section we consider some relations between ``adjoitness'' and ``monoidality'', the most essential among which is the {\it bialgebra axiom}, which seemingly is new.

For further reference, introduce some notations (and recall some very basic definitions and facts) related with adjoint functors.

Let $L\colon\mathscr{A}\rightleftarrows\mathscr{B}\colon R$ be two functors. They are called adjoint to each other, with $L$ the left adjoint and $R$ the right adjoint, when
\begin{equation}\label{adj0}
\Mor_\mathscr{B}(LX,Y)\simeq \Mor_\mathscr{A}(X,RY)
\end{equation}
where ``$\simeq$'' here means ``isomorphic as bifunctors $\mathscr{A}^\opp\times \mathscr{B}\to\Sets$''.

This gives rise to maps of functors $\epsilon\colon LR\to\Id_\mathscr{B}$ and $\eta\colon \Id_{\mathscr{A}}\to RL$ such that the compositions
\begin{equation}\label{adj}
\begin{aligned}
\ &L\xrightarrow{L\circ \eta}LRL\xrightarrow{\epsilon\circ L}L\\
&R\xrightarrow{\eta\circ R}RLR\xrightarrow{R\circ \epsilon}R
\end{aligned}
\end{equation}
are identity maps of the functors.

The inverse is true: given maps of functors $\epsilon\colon LR\to\Id_\mathscr{B}$ and $\eta\colon \Id_{\mathscr{A}}\to RL$, obeying \eqref{adj}, gives rise to the isomorphism of bifunctors, that is, to adjoint equivalence (see [ML], Section IV.1, Theorems 1 and 2).

In particular, the case of {\it adjoint equivalence} is the case when $\epsilon\colon LR\to \Id_\mathscr{B}$ and $\eta\colon\Id_\mathscr{A}\to RL$ are {\it isomorphisms of functors}. In this case, setting $\epsilon_1=\eta^{-1}$ and $\eta_1=\epsilon^{-1}$, we obtain another adjunction, with $L$ the {\it right} adjoint and $R$ the {\it left adjoint}.

Let $\phi\in \Mor_\mathscr{B}(LX,Y)$. The following explicit formula for its adjoint $\psi\in \Mor_\mathscr{A}(X,RY)$ will be useful:
\begin{equation}\label{adj2}
X\xrightarrow{\eta}RLX\xrightarrow{R(\phi)}RY
\end{equation}
and analogously for the way back:
\begin{equation}\label{adj3}
LX\xrightarrow{L(\psi)}LRY\xrightarrow{\epsilon}Y
\end{equation}
(see [ML], Section IV.1).

\subsection{\sc (Co)lax-monoidal structures and adjoint functors}\label{section12}
The game starts up with the following simple lemma, due to [SchS03].
\begin{lemma}\label{lemma11}
Let $\mathscr{C}$ and $\mathscr{D}$ be two strict monoidal categories, and
let $F\colon\mathscr{C}\rightleftarrows \mathscr{D}\colon G$ be a pair of adjoint functors, with $F$ the left adjoint. Then any colax-monoidal structure on $F$ induces a lax-monoidal structure on $G$, and vise versa. These two assignments are inverse to each other. {\rm (See Sections \ref{colax},\ref{lax} for (co)lax-monoidal functors).}
\end{lemma}

\begin{proof}
Denote by $c_F$ a colax-monoidal structure on $F$, and by $\ell_G$ a lax-monoidal structure on $G$. The assignment $c_F\rightsquigarrow \ell_G$ is defined as the adjoint to
\begin{equation}\label{colax2lax}
F(GX\otimes GY)\xrightarrow{c_F}FGX\otimes FGY\xrightarrow{\epsilon\otimes \epsilon}X\otimes Y
\end{equation}
The back assignment $\ell_G\rightsquigarrow c_F$ is defined as the adjoint to
\begin{equation}\label{lax2colax}
X\otimes Y\xrightarrow{\eta\otimes\eta}GFX\otimes GFY\xrightarrow{\ell_G}G(FX\otimes FY)
\end{equation}
We use the explicit formulas for the adjoint functors \eqref{adj2},\eqref{adj3} followed by \eqref{adj} to prove that these two assignments are inverse to each other.
\end{proof}

\begin{lemma}\label{lemma12}
Let $\mathscr{C}$ and $\mathscr{D}$ be two strict monoidal categories, and let $F\colon \mathscr{C}\rightleftarrows \mathscr{D}\colon G$ be an equivalence of the underlying categories. Then given a pair $(c_F,\ell_F)$ where $c_F$ is a colax-monoidal structure on $F$, $\ell_F$ is a lax-monoidal structure on $F$, one can assign to it a pair $(c_G,\ell_G)$ of analogous structures on $G$. If $\mathscr{C}$ and $\mathscr{D}$ are symmetric monoidal, and if the pair $(c_F,\ell_F)$ satisfies the  bialgebra axiom {\rm (see Section \ref{bialg})}, the pair $(c_G,\ell_G)$ satisfies the  bialgebra axiom as well.
\end{lemma}

\begin{proof}
Suppose $(c_F,\ell_F)$ are done. We firstly write down the formulas for $\ell_G$ and $c_G$.
\\[5pt]
Formula for $\ell_G$:
\begin{equation}\label{lg}
GX\otimes GY\xrightarrow{\eta}GF(GX\otimes GY)\xrightarrow{c_F}G(FGX\otimes FGY)\xrightarrow{\epsilon\otimes\epsilon}G(X\otimes Y)
\end{equation}
Formula for $c_G$:
\begin{equation}\label{cg}
G(X\otimes Y)\xrightarrow{\epsilon^{-1}\otimes\epsilon^{-1}}G(FGX\otimes FGY)\xrightarrow{\ell_F}GF(GX\otimes GY)\xrightarrow{\eta^{-1}}GX\otimes GY
\end{equation}
When we now write down the bialgebra axiom diagram (see Section \ref{bialg}) for $(c_G,\ell_G)$ we see due to cancelations of $\epsilon$ with $\epsilon^{-1}$ and of $\eta$ with $\eta^{-1}$, that the diagram is commutative as soon as the diagram for $(c_F,\ell_F)$ is.
\end{proof}

\begin{lemma}\label{lemma13}
Let $\mathscr{C}$ and $\mathscr{D}$ be two strict symmetric monoidal categories, and let $F\colon \mathscr{C}\to\mathscr{D}$ admits a colax-monoidal structure $c_F$ and a lax-monoidal structure $\ell_F$ such that:
\begin{itemize}
\item[i.] the both compositions
$$
F(X\otimes Y)\xrightarrow{c_F}F(X)\otimes F(Y)\xrightarrow{\ell_F}F(X\otimes Y)
$$
and
$$
F(X)\otimes F(Y)\xrightarrow{\ell_F}F(X\otimes Y)\xrightarrow{c_F}F(X)\otimes F(Y)
$$
are equal to the identity maps, for any $X,Y\in\Ob\mathscr{C}$,
\item[ii.] the both maps $c_F$ and $\ell_F$ are symmetric.
\end{itemize}
Then the pair $(c_F,\ell_F)$  satisfies the bialgebra axiom {\rm (see Section \ref{bialg})}.
\end{lemma}
It is clear.
\qed

\begin{coroll}\label{coroll14}
Let $\mathscr{C}$ and $\mathscr{D}$ be two strict symmetric monoidal categories, and let
$F\colon\mathscr{C}\rightleftarrows \mathscr{D}\colon G$ be an equivalence of categories. Suppose the left adjoint functor $F$ is symmetric monoidal. Then the corresponding pair $(c_G,\ell_G)$ of (co)lax-monoidal structures on the right adjoint $G$ obeys the bialgebra axiom. {\rm (The analogous is true when $G$ is symmetric monoidal, for (co)lax-monoidal structures on $F$).}
\end{coroll}
Follows directly from Lemmas \ref{lemma11}, \ref{lemma12}, and \ref{lemma13}.

\subsection{\sc Categories of monoids}\label{section13}
Let $\mathscr{C}$ be a symmetric strict monoidal category. Denote by $\Mon\mathscr{C}$ the category of strict associative monoids with unit in $\mathscr{C}$; it is again a symmetric monoidal category.

\begin{lemma}\label{lemma15}
Let $\mathscr{C}$ and $\mathscr{D}$ be strict symmetric monoidal categories, and let $F\colon \mathscr{C}\to \mathscr{D}$ be a functor. Any {\rm symmetric} lax-monoidal structure $\ell_F$ on the functor $F$ defines a functor $F^\mon\colon \Mon\mathscr{C}\to\Mon\mathscr{D}$, together with a lax-monoidal structure $\ell_F^\mon$ on $F^\mon$.
\end{lemma}

\qed

\begin{lemma}\label{lemma16}
In the assumptions of  Lemma \ref{lemma15}, suppose
$F$ admits a colax-monoidal structure $c_F$, compatible with $\ell_F$ by the bialgebra axiom {\rm (see Section \ref{bialg})}. Then one can define a colax-monoidal structure $c^\mon_F$ on $F^\mon$, and, moreover, $(c_F^\mon,\ell_F^\mon)$ are compatible by the bialgebra axiom.
\end{lemma}

\qed

\section{\sc The Dold-Kan correspondence}
We use the following notations:

$\mathscr{C}(\mathbb{Z})$ is the category of unbounded complexes of abelian groups, $\mathscr{C}(\mathbb{Z})^+$ (resp., $\mathscr{C}(\mathbb{Z})^-$) are the full subcategories of $\mathbb{Z}_{\ge 0}$-graded (resp., $\mathbb{Z}_{\le 0}$-graded) complexes. The category of abelian groups placed in degree 0 (with zero differential) is denoted by $\mathscr{M}od(\mathbb{Z})$, thus, $\mathscr{M}od(\mathbb{Z})=\mathscr{C}(\mathbb{Z})^-\cap \mathscr{C}(\mathbb{Z})^+$.

\subsection{\sc}
The Dold-Kan correspondence is the following theorem:
\begin{theorem}[Dold-Kan correspondence]
There is an adjoint equivalence of categories
$$
N\colon \mathscr{M}od(\mathbb{Z})^\Delta\rightleftarrows \mathscr{C}(\mathbb{Z})^-\colon\Gamma
$$
where $N$ is the functor of normalized chain complex (which is isomorphic to the Moore complex).
\end{theorem}
We refer to [W], Section 8.4, and [SchS03], Section 2, which both contain excellent treatment of this Theorem.

The both categories $\mathscr{M}od(\mathbb{Z})^\Delta$ and $\mathscr{C}(\mathbb{Z})^-$ are {\it symmetric monoidal} in natural way.
However, neither of functors $N$ and $\Gamma$ is monoidal.

There is a colax-monoidal structure on $N$, called {\it the Alexander-Whitney map} $AW\colon N(A\otimes B)\to NA\otimes NB$ and a lax-monoidal structure on $N$, called {\it the shuffle map} $\nabla\colon N(A)\otimes N(B)\to N(A\otimes B)$. 

Recall the explicit formulas for them.

The Alexander-Whitney colax-monoidal map $AW\colon N(A\otimes B)\to N(A)\otimes N(B)$ is defined as
\begin{equation}
AW(a^k\otimes b^k)=\sum_{i+j=k}d_\fin^ia^k\otimes d_0^jb^k
\end{equation}
where $d_0$ and $d_\fin$ are the first and the latest simplicial face maps. 

The Eilenberg-MacLane shuffle lax-monoidal map $\nabla\colon N(A)\otimes N(B)\to N(A\otimes B)$ is defined as
\begin{equation}
\nabla(a^k\otimes b^\ell)=\sum_{\substack{(k,\ell)\text{-shuffles }(\alpha,\beta)}}(-1)^{(\alpha,\beta)}S_\beta a^k\otimes S_\alpha b^\ell
\end{equation}
where 
$$
S_\alpha=s_{\alpha_k}\dots s_{\alpha_1}
$$
and 
$$
S_\beta=s_{\beta_\ell}\dots s_{\beta_1}
$$
Here $s_i$ are simplicial degeneracy maps, $\alpha=\{\alpha_1<\dots <\alpha_k\}$, $\beta=\{\beta_1<\dots<\beta_\ell\}$, $\alpha,\beta\subset [0,1,\dots, k+\ell-1]$, $\alpha\cap\beta=\varnothing$.

Let us summarize their properties in the following Proposition, see [SchS03], Section 2, and references therein, for a proof.
\begin{prop}\label{before}
The colax-monoidal Alexander-Whitney and the lax-monoidal shuffle structures on the functor $N$ enjoy the following properties:
\begin{itemize}
\item[1.] the composition
$$
NA\otimes NB\xrightarrow{\nabla}N(A\otimes B)\xrightarrow{AW}NA\otimes NB
$$
is the identity,
\item[2.] the composition
$$
N(A\otimes B)\xrightarrow{AW}NA\otimes NB\xrightarrow{\nabla}N(A\otimes B)
$$
is naturally chain homotopic to the identity,
\item[3.] the shuffle map $\nabla$ is symmetric,
\item[4.] the Alexander-Whitney map $AW$ is symmetric up to a natural chain homotopy.
\end{itemize}
\end{prop}

\begin{coroll}
The pair $(AW,\nabla)$ of (co)lax-monoidal structures on the functor $N$ obeys the bialgebra axiom {\rm (see Section \ref{bialg})} {\rm up to a chain homotopy}.
\end{coroll}

\subsection{\sc Main Theorem}
Our Main Theorem \ref{main} says the pair $(AW,\nabla)$ in fact obeys the bialgebra axiom on the nose, not only up to a homotopy.

\begin{theorem}\label{main}
The pair $(AW,\nabla)$ on the chain complex functor $C\colon \mathscr{M}od(\mathbb{Z})^\Delta\to C(\mathbb{Z})^-$ obeys the bialgebra axiom. Consequently, the normalized chain complex functor $N$ obeys the bialgebra axiom as well.
\end{theorem}
We start to prove the theorem.

We compute the two arrows $C(X\otimes Y)\otimes C(Z\otimes W)\rightrightarrows C(X\otimes Z)\otimes C(Y\otimes W)$, which figure out in the bialgebra axiom, namely, \eqref{bialgebra1} and \eqref{bialgebra2}. Let $x\in X_k$, $y\in Y_k$, $z\in Z_\ell$, $w\in W_\ell$.

\smallskip

\ \\
{\bf Computation of \eqref{bialgebra2}:}

\begin{equation}\label{computation11}
\begin{aligned}
\ &x^k\otimes y^k\xrightarrow{AW}\sum_{i+j=k}d_\fin^ix^k\otimes d_0^jy^k\\
&z^\ell\otimes w^\ell\xrightarrow{AW}\sum_{a+b=\ell}d_\fin^az^\ell\otimes d_0^bw^\ell
\end{aligned}
\end{equation}

\begin{equation}\label{computation12}
\begin{aligned}
\ &\left(\sum_{i+j=k}d_\fin^ix^k\otimes d_0^jy^k\right)\otimes\left(\sum_{a+b=\ell}d_\fin^az^\ell\otimes d_0^bw^\ell\right)\xrightarrow{\sigma}\\
&\sum_{\substack{i+j=k\\a+b=\ell}}(-1)^{(k-j)(\ell-a)}\left(d_\fin^ix^k\otimes d_\fin^az^\ell\right)\otimes\left(d_0^jy^k\otimes d_0^bw^\ell\right)
\end{aligned}
\end{equation}

The final application of the shuffle map to the r.h.s. of \eqref{computation12} gives:

\begin{equation}\label{computation13}
\sum_{\substack{(k-i, \ell-a)\text{-shuffles }(\mu,\nu)\\
(k-j,\ell-b)\text{-shuffles }(\mu^\prime,\nu^\prime)}}
(-1)^{(\mu,\nu)+(\mu^\prime,\nu^\prime)}\sum_{\substack{i+j=k\\a+b=\ell}}(-1)^{(k-j)(\ell-a)}\left(S_\nu d_\fin^i x^k\otimes S_\mu d_\fin^a z^\ell\right)\otimes\left(S_{\nu^\prime}d_0^jy^k\otimes S_{\mu^\prime}d_0^b w^\ell\right)
\end{equation}
\\
{\bf Computation of \eqref{bialgebra1}:}
\begin{equation}\label{computation21}
(x^k\otimes y^k)\otimes (z^\ell\otimes w^\ell)\xrightarrow{\nabla}\sum_{\substack{(k,\ell)\text{-shuffles }(\alpha,\beta)}}
(-1)^{(\alpha,\beta)}\left(S_\beta x^k\otimes S_\beta y^k\right)\otimes\left(S_\alpha z^\ell\otimes S_\alpha w^\ell\right)
\end{equation}
The permutation $\sigma$ maps the r.h.s. of \eqref{computation21} to
\begin{equation}\label{computation22}
\sum_{\substack{(k,\ell)\text{-shuffles }(\alpha,\beta)}}(-1)^{(\alpha,\beta)}\left(S_\beta x^k\otimes S_\alpha z^\ell\right)\otimes\left(S_\beta y^k\otimes S_\alpha w^\ell\right)
\end{equation}
Finally, the Alexander-Whitney map $AW$ maps the r.h.s. of \eqref{computation22} to
\begin{equation}\label{computation23}
\sum_{\substack{s+t=k+\ell}}\sum_{\substack{(k,\ell)\text{-shuffles }(\alpha,\beta)}}(-1)^{(\alpha,\beta)}\left(d_\fin^sS_\beta x^k\otimes d_\fin^s S_\alpha z^\ell\right)\otimes\left(d_0^tS_\beta y^k\otimes d_0^t S_\alpha w^\ell\right)
\end{equation}

We need to prove that the expressions in \eqref{computation13} and \eqref{computation23} coincide, mod out degenerate simplices. (In fact, as we will see, these expressions coincide on the nose).

We need to move $d_\fin^s$ to the right over $S_\beta$ and $S_\alpha$, and analogously to move $d_0^t$ to the right over $S_\alpha$ and $S_\beta$, in the expression \eqref{computation23}. 
According to \eqref{simplicial}, one has
\begin{equation}\label{dfin}
\begin{aligned}
\ &d_\fin s_\maxx=1\\
&d_\fin s_j=s_jd_\fin \ \ \text{for $j<\maxx$}
\end{aligned}
\end{equation}
(where in the first line $\maxx=\fin-1$, and $\fin$ in the r.h.s. of the second line is for 1 less than $\fin$ in the l.h.s.), and analogously
\begin{equation}\label{d0}
\begin{aligned}
\ &d_0 s_0=1\\
&d_0 s_j=s_{j-1} d_0 \ \ \ \text{for $j>0$}
\end{aligned}
\end{equation}

Let $(\alpha,\beta)$ be a $(k,\ell)$-shuffle, $\alpha=\{\alpha_1<\dots <\alpha_k\}$, $\beta=\{\beta_1<\dots<\beta_\ell\}$, $\alpha,\beta\subset [0,1,\dots, k+\ell-1]$, $\alpha\cap\beta=\varnothing$.

Let $s+t=k+\ell$.

Compute $d_0^tS_\alpha$ and $d_0^tS_\beta$.

From \eqref{d0} we see that
\begin{equation}
d_0^tS_\alpha=S_{\alpha^\prime}d_0^{t-\sharp\alpha^{< t}}
\end{equation}

\begin{equation}
d_0^tS_\beta=S_{\beta^\prime}d_0^{t-\sharp\beta^{<t}}
\end{equation}
for some $\alpha^\prime$, $\beta^\prime$.

Analogously, from \eqref{dfin} we see that

\begin{equation}
d_\fin^sS_\alpha=S_{\alpha^{\prime\prime}}d_\fin^{s-\sharp\alpha^{\ge t}}
\end{equation}

\begin{equation}
d_\fin^sS_\beta=S_{\beta^{\prime\prime}}d_\fin^{s-\sharp\beta^{\ge t}}
\end{equation}
for some $\alpha^{\prime\prime}$, $\beta^{\prime\prime}$.

Here we use notation
\begin{equation}
\alpha^{<t}=\{p|\alpha_p<t\}
\end{equation}
and analogously for $\alpha^{\ge t}$, $\beta^{<t}$, $\beta^{\ge t}$. 

Now set
\begin{equation}
\begin{aligned}
\ &a= s-\sharp \alpha^{\ge t}  \\
&b=t-\sharp \alpha^{<t}    \\
&i= s-\sharp\beta^{\ge t}   \\
&j=  t-\sharp\beta^{<t}  
\end{aligned}
\end{equation}

Then $i+j=k$, $a+b=\ell$, as one immediately sees. 

We have
\begin{lemma}
The pairs $(\alpha^\prime,\beta^\prime)$ and $(\alpha^{\prime\prime},\beta^{\prime\prime})$ are shuffles, for any $s,t$ such that $s+t=k+\ell$.
\end{lemma}
It is clear.\qed

We established, mod out the sign, a 1-1 correspondence between the summand of \eqref{computation13} and \eqref{computation23}.
The signs is a straightforward check, which is left to the reader.

\section{\sc Diagrams}\label{diagrams}

\subsection{\sc Colax-monoidal structure on a functor}\label{colax}
\begin{defn}[{\rm Colax-monoidal functor}]\label{colax_intro}{\rm
Let $\mathscr{M}_1$ and $\mathscr{M}_2$ be two strict associative monoidal categories. A functor $F\colon\mathscr{M}_1\to \mathscr{M}_2$ is called {\it colax-monoidal} if there is a map of bifunctors $\beta_{X,Y}\colon F(X\otimes Y)\to F(X)\otimes F(Y)$ and a morphism $\alpha\colon F(1_{\mathscr{M}_1})\to 1_{\mathscr{M}_2}$ such that:

(1): for any three objects $X,Y,Z\in\Ob(\mathscr{M}_1)$, the diagram
\begin{equation}\label{colaxdiagram}
\xymatrix{
&F(X\otimes Y)\otimes F(Z)\ar[rd]^{\beta_{X,Y}\otimes\id_{F(Z)}}\\
F(X\otimes Y\otimes Z)\ar[ru]^{\beta_{X\otimes Y,Z}}\ar[rd]_{\ \ \ \ \ \ \ \beta_{X,Y\otimes Z}}&& F(X)\otimes F(Y)\otimes F(Z)\\
&F(X)\otimes F(Y\otimes Z)\ar[ru]_{\id_{F(X)}\otimes\beta_{Y,Z}}
}
\end{equation}
is commutative. The functors $\beta_{X,Y}$ are called the {\it colax-monoidal maps}.

(2): for any $X\in\Ob\mathscr{M}_1$ the following two diagrams are commutative
\begin{equation}
\begin{aligned}
\ &\xymatrix{
F(1_{\mathscr{M}_1}\otimes X)\ar[d]\ar[r]^{\beta_{1,X}}& F(1_{\mathscr{M}_1})\otimes F(X)\ar[d]^{\alpha\otimes\id}\\
F(X)&      1_{\mathscr{M}_2}\otimes F(X)\ar[l]
} &
\xymatrix{
F(X\otimes 1_{\mathscr{M}_1})\ar[d]\ar[r]^{\beta_{X,1}}& F(X)\otimes F(1_{\mathscr{M}_1})\ar[d]^{\id\otimes \alpha}\\
F(X)&     F(X)\otimes 1_{\mathscr{M}_2}\ar[l]
}
\end{aligned}
\end{equation}
}
\end{defn}

\subsection{\sc Lax-monoidal structure on a functor}\label{lax}
\begin{defn}[{\rm Lax-monoidal functor}]\label{lax_intro}{\rm
Let $\mathscr{M}_1$ and $\mathscr{M}_2$ be two strict associative monoidal categories. A functor $G\colon\mathscr{M}_1\to \mathscr{M}_2$ is called {\it lax-monoidal} if there is a map of bifunctors $\gamma_{X,Y}\colon G(X)\otimes G(Y)\to G(X\otimes Y)$ and a morphism $\kappa\colon 1_{\mathscr{M}_2}\to G(1_{\mathscr{M}_1})$ such that:

(1): for any three objects $X,Y,Z\in\Ob(\mathscr{M}_1)$, the diagram
\begin{equation}\label{colaxdiagram}
\xymatrix{
&G(X\otimes Y)\otimes G(Z)\ar[dl]_{\gamma_{X\otimes Y,Z}}\\
G(X\otimes Y\otimes Z)&& G(X)\otimes G(Y)\otimes G(Z)\ar[ul]_{\gamma_{X, Y}\otimes\id_{G(Z)}}\ar[dl]^{\id_{G(X)}\otimes\gamma_{Y,Z}}\\
&G(X)\otimes G(Y\otimes Z)\ar[ul]^{\gamma_{X,Y\otimes Z}}
}
\end{equation}
is commutative. The functors $\gamma_{X,Y}$ are called the {\it lax-monoidal maps}.

(2): for any $X\in\Ob\mathscr{M}_1$ the following two diagrams are commutative
\begin{equation}
\begin{aligned}
\ &\xymatrix{
F(1_{\mathscr{M}_1}\otimes X)\ar[d]& F(1_{\mathscr{M}_1})\otimes F(X)\ar[l]_{\gamma_{1,X}}\\
F(X)&      1_{\mathscr{M}_2}\otimes F(X)\ar[u]_{\kappa\otimes\id}\ar[l]
} &
\xymatrix{
F(X\otimes 1_{\mathscr{M}_1})\ar[d]& F(X)\otimes F(1_{\mathscr{M}_1})\ar[l]_{\gamma_{X,1}}\\
F(X)&     F(X)\otimes 1_{\mathscr{M}_2}\ar[u]_{\id\otimes \kappa}\ar[l]
}
\end{aligned}
\end{equation}
}
\end{defn}

\subsection{\sc Bialgebra axiom}\label{bialg}
This axiom, expressing a compatibility between the lax-monoidal and colax-monoidal structures on a functor between  {\it symmetric} monoidal categories, seems to be new.
\begin{defn}[Bialgebra axiom]{\rm
Suppose there are given both colax-monoidal and lax-monoidal structures on a functor $F\colon \mathscr{C}\to\mathscr{D}$, where $\mathscr{C}$ and $\mathscr{D}$ are strict {\it symmetric} monoidal categories.
Denote these structures by $c_F(X,Y)\colon F(X\otimes Y)\to F(X)\otimes F(Y)$, and $l_F\colon F(X)\otimes F(Y)\to F(X\otimes Y)$.
We say that the pair $(l_F,c_F)$ satisfies the bialgebra axiom, if for any for objects $X,Y,Z,W\in\Ob\mathscr{C}$, the following two morphisms
$F(X\otimes Y)\otimes F(Z\otimes W)\to F(X\otimes Z)\otimes F(Y\otimes W)$
coincide:
\begin{equation}\label{bialgebra1}
\begin{aligned}
\ &F(X\otimes Y)\otimes F(Z\otimes W)\xrightarrow{l_F}F(X\otimes Y\otimes Z\otimes W)\xrightarrow{F(\id\otimes \sigma\otimes \id)}\\
&F(X\otimes Z\otimes Y\otimes W)\xrightarrow{c_F}F(X\otimes Z)\otimes F(Y\otimes W)
\end{aligned}
\end{equation}
and
\begin{equation}\label{bialgebra2}
\begin{aligned}
\ &F(X\otimes Y)\otimes F(Z\otimes W)\xrightarrow{c_F\otimes c_F}F(X)\otimes F(Y)\otimes F(Z)\otimes F(W)\xrightarrow{\id\otimes \sigma \otimes \id}\\
& F(X)\otimes F(Z)\otimes F(Y)\otimes F(W)\xrightarrow{l_F\otimes l_F}F(X\otimes Z)\otimes F(Y\otimes W)
\end{aligned}
\end{equation}
where $\sigma$ denotes the symmetry morphisms in $\mathscr{C}$ and in $\mathscr{D}$.

Thus, the commutative diagram, expressing the bialgebra axiom, is
\begin{equation}
\xymatrix{
F(X\otimes Y)\otimes F(Z\otimes W)\ar@/^2pc/[rr]^{\eqref{bialgebra1}}\ar@/_2pc/[rr]_{\eqref{bialgebra2}}&&F(X\otimes Z)\otimes F(Y\otimes W)
}
\end{equation}
}
\end{defn}

\bigskip

\noindent {\sc Max-Planck Institut f\"{u}r Mathematik, Vivatsgasse 7, 53111 Bonn,\\
GERMANY}

\bigskip

\noindent {\em E-mail address\/}: {\tt borya$\_$port@yahoo.com}

\end{document}